\PassOptionsToPackage{pdfpagelabels=false}{hyperref} 
\documentclass[english,a4paper,11pt]{article}

\usepackage{fullpage}

\usepackage{amsxtra, amsfonts, amssymb, amstext}
\usepackage{amsthm}
\usepackage{booktabs}
\usepackage{longtable}
\usepackage{nicefrac}
\usepackage{xspace}
\usepackage[noadjust]{cite}
\usepackage{url}
\usepackage{graphics,color}
\usepackage[colorlinks]{hyperref}
\definecolor{linkblue}{rgb}{0.1,0.1,0.8}
\hypersetup{colorlinks=true,linkcolor=linkblue,filecolor=linkblue,urlcolor=linkblue,citecolor=linkblue}
\usepackage[algo2e,ruled,vlined,linesnumbered]{algorithm2e}

\usepackage{graphicx}
\usepackage{wrapfig}

\newtheorem{theorem}{Theorem}
\newtheorem{lemma}[theorem]{Lemma}

\newtheorem{definition}[theorem]{Definition}

\newcommand{\ignore}[1]{}

\allowdisplaybreaks[1]

\newcommand{\N}{\mathbb{N}}
\newcommand{\R}{\mathbb{R}}
\newcommand{\Z}{\mathbb{Z}}

\renewcommand{\epsilon}{\varepsilon}

\begin{document}

\title{Probabilistic Lower Bounds for the Discrepancy of Latin Hypercube Samples}

\author{Benjamin Doerr$^1$ \and Carola Doerr$^2$ \and Michael Gnewuch$^3$}

\date{
$^1$\'Ecole Polytechnique, LIX - UMR 7161, CS35003, 91120 Palaiseau, France\\
$^2$Sorbonne Universit\'es, UPMC Univ Paris 06, CNRS, LIP6 UMR 7606, 4 place Jussieu 75005 Paris, France\\
$^3$Christian-Albrechts-Universit\"at Kiel, Mathematisches Seminar, Ludewig-Meyn-Str. 4, 24098 Kiel, Germany
}
\maketitle

\abstract{
We provide probabilistic lower bounds for the star discrepancy of Latin hypercube samples. 
These bounds are sharp in the sense that they match the recent probabilistic upper bounds for the star discrepancy of Latin hypercube samples proved in [M.~Gnewuch, N.~Hebbinghaus. {\sl Discrepancy bounds for a class of negatively dependent random points including Latin hypercube samples}. Preprint 2016.]. Together, this result and our work implies that the discrepancy of Latin hypercube samples differs at most by constant factors from the discrepancy of uniformly sampled point sets.
}

\section{Introduction}

Discrepancy measures are well established and play an important role in fields like computer graphics,
experimental design, pseudo-random number generation, stochastic programming, numerical integration or, more general, stochastic simulation.

The prerelevant and most intriguing discrepancy measure is arguably the \emph{star discrepancy}, which is defined in the following way:

Let $P\subset [0,1)^d$ be an $N$-point set. (We always understand an ``$N$-point set'' as a ``multi-set'', i.e., it consists of $N$ points, but these points do not have to be pairwise different.) We define the \emph{local discrepancy} of $P$ with respect to a Lebesgue-measurable test set $T\subseteq [0,1)^d$ by
\begin{equation*}
D_N(P,T) := \bigg| \frac{1}{N} |P\cap T| - \lambda^d(T) \bigg|,
\end{equation*}
where $|P\cap T|$ denotes the size of the finite set $P\cap T$ (again understood as a multi-set) and $\lambda^d$ the $d$-dimensional Lebesgue measure on $\R^d$.
For  vectors $x= (x_1,x_2,\ldots, x_d)$, $y= (y_1, y_2, \ldots, y_d) \in \R^d$ 
we write 
\begin{equation*}
[x,y) := \prod^d_{j=1} [x_j,y_j) = \{z\in \R^d \,|\, x_j \le z_j < y_j \hspace{1ex}\text{for $j = 1,\ldots, d$}\}.
\end{equation*}
The \emph{star discrepancy} of $P$ is then given by
\begin{equation*}
D^*_N(P) := \sup_{y\in [0,1]^d} D_N(P,[0,y)).
\end{equation*}
We will refer to the sets $[0,y)$, $y\in [0,1]^d$, as \emph{anchored test boxes}.

The star discrepancy is intimately related to quasi-Monte Carlo integration via the Koksma-Hlawka inequality: For every $N$-point set $P\subset [0,1)^d$ we have
\begin{equation*}
\left| \int_{[0,1)^d} f(x) \,d\lambda^d(x) - \frac{1}{N} \sum_{p\in P} f(p) \right| 
\le D_N^*(P) {\rm Var}_{\rm HK}(f), 
\end{equation*}
where ${\rm Var}_{\rm HK}(f)$ denotes the variation in the sense of Hardy and Krause see, e.g., \cite{Nie92}. 
The Koksma-Hlawka inequality is sharp, see again \cite{Nie92}.
(An alternative version of the Koksma-Hlawka inequality can be found in \cite{HSW04}; it says that the worst-case error 
of equal-weight cubatures based on a set of integration points $P$ over the norm unit ball of some Sobolev space is exactly the star discrepancy of $P$.)
The Koksma-Hlawka inequality shows that equal-weight cubatures based on integration points with small star discrepancy yield  small  integration errors. (Deterministic equal-weight cubatures are commonly called \emph{quasi-Monte Carlo algorithms}; for a recent survey we refer to \cite{DKS13}.)  
For the very important task of high-dimensional integration, which occurs, e.g., in mathematical finance, physics or quantum chemistry, it is therefore of interest to know sharp bounds for the smallest achievable star discrepancy
and to be able to construct integration points that satisfy those bounds. To avoid the ``curse of dimensionality'' it is crucial that such bounds scale well with respect to the dimension.

The best known upper and lower bounds for the smallest achievable star discrepancy with explicitly given dependence on the number of sample points~$N$ as well as on the dimension~$d$ are of the following form: For all $d,N \in \N$ there exists
an $N$-point set $P\subset [0,1)^d$ satisfying
\begin{equation}\label{hnww01}
D^*_N(P) \le C \sqrt{\frac{d}{N}}
\end{equation}
for some universal constant $C>0$, while for all $N$-point sets $Q\subset [0,1)^d$ it holds that
\begin{equation}\label{hin04}
D^*_N(Q) \ge \min \left\{c_0, c \frac{d}{N} \right\},
\end{equation}
where $c_0,c \in (0,1]$ are suitable constants. The upper bound (\ref{hnww01}) was proved 
by Heinrich et al. \cite{HNWW01} without providing an estimate for the universal constant $C$.
The first estimate for this constant was given by Aistleitner \cite{Ais11}; he showed that $C\le 9.65$. This estimate has recently been improved to $C \le 2.5287$ in~\cite{GH16}. 
All these results are based on probabilistic arguments and do not provide an explicit point construction that satisfies~\eqref{hnww01}.
The lower bound (\ref{hin04}) was established by Hinrichs \cite{Hin04}.
Observe that there is a gap between the upper bound (\ref{hnww01}) and the lower bound (\ref{hin04}). 
In \cite[Problem~1 \& 2]{Hei03} Heinrich asked the following two questions:
\begin{itemize}
\item[(a)] Does any of the various known constructions of low discrepancy point sets
satisfy an estimate like (\ref{hnww01}) or at least some slightly weaker estimates? 
\item[(b)] What are the correct sharp bounds for the smallest achievable star discrepancy? 
\end{itemize}
It turned out that these two questions are very difficult to answer.

To draw near an answer, it was proposed in \cite{GH16} to study the following related questions:
\begin{itemize}
\item[(c)] What kind of randomized point constructions 
satisfy~(\ref{hnww01}) in expectation and/or with high probability? 
\item[(d)] Can it even be shown, by probabilistic constructions, that the upper bound~\eqref{hnww01} is too pessimistic? 
\end{itemize}
As mentioned, the upper bound~\eqref{hnww01} was proved via probabilistic arguments. Indeed,
Monte Carlo points, i.e., 
independent random points uniformly distributed in $[0,1)^d$, satisfy this bound with high probability.
In~\cite{Doe13} it was rigorously shown that the star discrepancy of Monte Carlo point sets $X$ behaves like the right hand side in (\ref{hnww01}). More precisely,  there exists a constant $K>0$ such that the expected star discrepancy of $X$ is bounded from below by
\begin{equation}\label{benji1}
{\bf {\rm E}} [D^*_N(X)] \ge K \sqrt{\frac{d}{N}}
\end{equation}
and additionally we have the probabilistic discrepancy bound
\begin{equation}\label{benji2}
{\bf {\rm P}}\! \left( D^*_N(X) < K\sqrt{\frac{d}{N}} \right) \le \exp(-\Omega(d)).
\end{equation}
The upper bound~\eqref{hnww01} is thus sharp for Monte Carlo points, showing that they cannot be employed to improve it.

What about other randomized point constructions? In \cite{GH16} it is shown that so-called Latin hypercube samples
satisfy the upper bound (\ref{hnww01}) with high probability, see Theorem~\ref{Theo_Nils_Michi} below.
In this note we show that this estimate is tight. More precisely, we prove that the bounds \eqref{benji1} and \eqref{benji2} for Monte Carlo point sets also apply to Latin hypercube samples.

\section{Probabilistic Discrepancy Bounds for Latin Hypercube Sampling}
\label{LHS}

For $N\in \N$ we denote the set $\{1,2,\ldots,N\}$ by $[N]$.
The definition of Latin hypercube sampling presented below was introduced by McKay, Beckman, 
and Conover \cite{MBC79} for the design of computer experiments.   

\begin{definition}\label{Def_LHS}
A \emph{Latin hypercube sample} (LHS) $(X_n)_{n\in [N]}$ in $[0,1)^d$ is of the form
\begin{equation*}
X_{n,j} = \frac{\pi_j(n) - u_{n,j}}{N},
\end{equation*}
where $X_{n,j}$ denotes the $j$th coordinate of $X_n$, $\pi_j$ is a permutation
of $[N]$ that is chosen uniformly at random, and $u_{n,j}$ obeys the uniform distribution on $[0,1)$.
The $d$ permutations $\pi_j$ and the $dN$ random variables $u_{n,j}$ are mutually independent. 
\end{definition}

The following result was proved in \cite{GH16}.

\begin{theorem}\label{Theo_Nils_Michi}
Let $d, N\in \N$, and let $X= (X_n)_{n\in[N]}$ be a Latin hypercube sample in $[0,1)^d$. 
Then for every $c>0$
\begin{equation*}
{\bf {\rm P}}\! \left(D^*_N(X) \leq c\sqrt{\frac{d}{N}} \, \right) \ge 1-\exp \left(-(1.6741\, c^2 -11.7042)\, d \right).
\end{equation*}
In particular, there exists a realization $P\subset [0,1)^d$ of $X$ such that
\begin{equation*}
D^*_N(P) \leq 2.6442 \cdot\sqrt{\frac{d}{N}}
\end{equation*}
and the probability that $X$ satisfies
\begin{equation*}
D^*_N(X) \leq 3\cdot\sqrt{\frac{d}{N}}
\hspace{3ex}\text{and}\hspace{3ex}
D^*_N(X) \leq 4\cdot\sqrt{\frac{d}{N}}
\end{equation*}
is at least $0.965358$ and $0.999999$, respectively. 
\end{theorem}

The result of our note complements the previous theorem and shows that it is sharp from a probabilistic point of view. 


\begin{theorem}\label{main_theo}
There exists a constant $K>0$ such that for all 
$d, N\in \N$ with $d \ge 2$ and $N \ge 1,600d$,  the discrepancy of a Latin hypercube sample $X = (X_n)_{n\in[N]}$ in $[0,1)^d$ satisfies
\begin{equation*}
{\bf {\rm E}} [D^*_N(X)] \ge K \sqrt{\frac{d}{N}}
\end{equation*}
and
\begin{equation*}
{\bf {\rm P}}\! \left( D^*_N(X) < K\sqrt{\frac{d}{N}} \right) \le \exp(-\Omega(d)).
\end{equation*}
\end{theorem}

We note that Theorem~\ref{main_theo} does not hold for $d=1$.
Indeed, it is easily verified that in dimension
$d=1$ we have $D^*_N(X) \le 1/N$ almost surely. 

\section{Proof of Theorem~\ref{main_theo}}
\label{PROOF}

We now list the results that we need to prove Theorem~\ref{main_theo}.
We will employ the fact that (under suitable conditions) the hypergeometric distribution resembles 
the binomial distribution. Let us make this statement more precise. 

Consider an urn that contains $N$ balls among which $W$ are white and $N-W$ are black.
Now we draw a random sample of size $n$. The number of white balls in the sample has the \emph{hypergeometric distribution} $H(N, W, n)$ if we sample without replacement and the \emph{binomial distribution} $B(n,p)$ with
$$p:=W/N$$ 
if we sample with replacement.
The deviation of both distributions can be measured by the \emph{total variation distance}
\begin{equation*}
\delta \big( H(N, W, n), B(n,p) \big) := \max_{A\subseteq \{0,1,\ldots,n\}} |H(N, W, n)(A) - B(n,p)(A)|.
\end{equation*}
The following theorem can be found in \cite[p.~1]{Kue98}; here we only need the upper bound, which is due to Ehm, see \cite{Ehm91}. 

\begin{theorem}\label{HypGeomBinom}
Let $n,N,W \in \N$ with $W,n\le N$ and let $p \in (0,1)$ such that $np(1-p) \ge 1$.
Then 
\begin{equation*}
\frac{1}{28} \frac{n-1}{N-1} \le \delta \big( H(N, W, n),B(n,p) \big) \le \frac{n-1}{N-1}.
\end{equation*}
\end{theorem}

Furthermore, we will make use of the following lemma from \cite{Doe13}.

\begin{lemma}\label{Lemma_Benji}
Let $n\ge 16$ and $1/n \le p \le 1/4$. Then
\begin{equation*}
B(n,p)\! \left( \left[0, np - \frac{1}{2} \sqrt{np} \right] \right) \ge \frac{3}{160}.
\end{equation*}  
\end{lemma}

Finally, we need the following Chernoff-Hoeffding bound 
for sums of independent Bernoulli random variables, 
see \cite{Hoe63}.
Recall that a Bernoulli random variable is simply a random variable that takes only values in $\{0,1\}$.
 
 \begin{theorem}\label{Chernov}
Let $k\in\N$, and  let $\xi_1, \ldots , \xi_k$ be independent (not necessarily identically distributed) Bernoulli random variables. 
Put $S:= \sum^k_{i=1} (\xi_i-{\rm E}[\xi_i])$.
Then we have  for all $t>0$ that
\begin{eqnarray}\label{Hoeffding_both}
{\rm P}\left( S < -tk \right)  \leq \exp\left(- 2t^2k\right).
\end{eqnarray}
\end{theorem} 

The Bernoulli random variables $\eta_i$, $i=1,\ldots, d$, that appear in our proof of Theorem \ref{main_theo}
are actually not independent; to cope with that we need the following lemma.

\begin{lemma}\label{Lemma_Coupling}
Let $k\in \N$ and $q\in (0,1)$. Let  $\xi_1,\ldots, \xi_k$ be independent Bernoulli random variables with ${\rm P}(\xi_j=1) = q$ for all $j\in [k]$, and let $\eta_1,\ldots, \eta_k$ be (not necessarily independent) Bernoulli random variables satisfying 
\begin{equation*}
{\rm P}(\eta_j =1 \,|\, \eta_1=v_1, \ldots,\eta_{j-1} = v_{j-1} ) \ge q
\hspace{3ex}\text{for all $j\in [k]$ and all $v\in \{0,1\}^{j-1}$.}
\end{equation*}
Then we have
\begin{equation}\label{coupling_statement}
{\rm P}\left( \sum_{i=1}^j \eta_i <t \right) \le {\rm P}\left( \sum_{i=1}^j \xi_i <t \right)
\hspace{3ex}\text{for all $j\in [k]$ and all $t>0$.}
\end{equation}
\end{lemma}

Since we do not know a proper reference for this lemma, we provide a proof. For a finite bit string $v\in\{0,1\}^j$ we put $|v|_1 := v_1 + \cdots + v_j$. 

\begin{proof}
We verify statement \eqref{coupling_statement} by induction on $j$. For $j=1$ statement \eqref{coupling_statement}
is true, since for $t\in (0,1]$ we have ${\rm P}(\eta_1 <t) \le 1-q = {\rm P}(\xi_1 <t)$ and for the trivial case $t>1$ we have
${\rm P}(\eta_1 <t)  =1 = {\rm P}(\xi_1 <t)$.

Now assume that statement \eqref{coupling_statement} is true for  $j\in [k-1]$. This gives for $t>0$
\begin{equation*}
\begin{split}
&{\rm P}\left( \sum_{i=1}^{j+1} \eta_i <t \right) = {\rm P}\left( \sum_{i=1}^{j} \eta_i < t-1 \right) + {\rm P}\left( \eta_{j+1} =0\,,\, \sum_{i=1}^{j} \eta_i \in [t-1, t) \right)\\
= \,&{\rm P}\left( \sum_{i=1}^{j} \eta_i < t-1 \right) + \sum_{\substack{v\in \{0,1\}^{j}\\ |v|_1 \in [t-1,t)}} {\rm P}\left( \eta_{j+1} =0\,|\, \eta_1 =v_1, \ldots,  \eta_j=v_j \right)\times \\
&\times {\rm P}\left( \eta_1 =v_1, \ldots,  \eta_j=v_j \right) \\
\le \,&{\rm P}\left( \sum_{i=1}^{j} \eta_i < t-1 \right) + (1-q) {\rm P}\left( \sum_{i=1}^{j} \eta_i \in [t-1, t) \right)\\
=  \,&q\,{\rm P}\left( \sum_{i=1}^{j} \eta_i < t-1 \right) + (1-q)\,{\rm P}\left( \sum_{i=1}^{j} \eta_i <t \right)\\
\le \,&q\,{\rm P}\left( \sum_{i=1}^{j} \xi_i < t-1 \right) + (1-q)\,{\rm P}\left( \sum_{i=1}^{j} \xi_i <t \right)\\
= \,&{\rm P}\left( \sum_{i=1}^{j} \xi_i < t-1 \right) + (1-q) {\rm P}\left( \sum_{i=1}^{j} \xi_i \in [t-1, t) \right)\\
= \,&{\rm P}\left( \sum_{i=1}^{j} \xi_i < t-1 \right) + {\rm P}\left( \xi_{j+1} =0\,,\, \sum_{i=1}^{j} \xi_i \in [t-1, t) \right)\\
= \,&{\rm P}\left( \sum_{i=1}^{j+1} \xi_i <t \right).
\end{split}
\end{equation*}
\end{proof}

For a given $N$-point set $P\subset [0,1)^d$ and a measurable set $B\subseteq [0,1)^d$ let us define the 
\emph{excess of points from $P$ in $B$} by 
\begin{equation*}
{\rm exc}(P,B) := |P \cap B| - N\lambda^d(B).
\end{equation*} 
For an arbitrary anchored test box $B$ we always have 
\begin{equation}\label{karo}
D^*_N(P) \ge D_N(P,B) \ge \frac{1}{N} {\rm exc}(P,B).
\end{equation}

\begin{proof}[Proof of Theorem \ref{main_theo}]
We adapt the proof approach of \cite[Theorem~1]{Doe13}  and construct recursively a random test box $B_d = B_d(X)$ that exhibits with high probability a (relatively) large excess of points ${\rm exc}(X,B_d)$. Due to \eqref{karo} this leads to a (relatively) large local discrepancy $D_N(X,B_d)$. Put $I:= [0, \lfloor N/4 \rfloor /N)$,
where $\lfloor N/4 \rfloor := \max\{z\in \Z \,|\, z \le N/4\}$. 
We start with $B_1 := I \times [0,1)^{d-1}$. Notice that there are exactly $\lfloor N/4 \rfloor$ points of $X$ inside the box $B_1$, implying ${\rm exc}(X, B_1) = 0$. 
The recursion step is as follows: Let $j\ge 2$ and assume we already have a test box $B_{j-1}$ that satisfies ${\rm exc}(X, B_{j-1}) \ge 0$ and is of the form 
\begin{equation*}
B_{j-1} := I \times \prod^{j-1}_{i=2}[0,x_i) \times [0,1)^{d-j+1},
\end{equation*}
where $x_i \in \{1- c/d, 1\}$ for $i=2,\ldots, j-1$ and $c$ is the largest value in $(1/84, 1/80]$ that ensures $Nc/d\in \N$. 
Observe that due to $N\ge 1600\,d$ we have $Nc/d\ge 20$ and $\lambda^d(B_1)= \lambda^1(I) \in (1/5, 1/4]$. 
Let 
\begin{equation*}
S_j:= [0,1)^{j-1} \times [1-c/d, 1) \times [0,1)^{d-j}
\hspace{3ex}\text{and}\hspace{3ex}
C_j:= 
B_{j-1}\cap S_j,
\end{equation*}
and put 
$$Y_j:= |X \cap C_j|.$$ Looking at Definition~\ref{Def_LHS} one sees easily that $Y_j$ has the hypergeometric distribution $H(N, W, n)$ with
\begin{equation*}
W := |X \cap B_{j-1}| \hspace{3ex}\text{and}\hspace{3ex}  n := |X\cap S_j| = N \frac{c}{d}.
\end{equation*}
Observe that 
\begin{equation}\label{stern}
\frac{1}{4} \ge \lambda^d(B_{j-1}) \ge \frac{1}{5}(1-c/d)^{d-2} \ge \frac{1}{5}(1-c/d)^d \ge \frac{1}{5}(1-c/2)^2 =: v \ge \frac{1}{6},
\end{equation}
and, due to ${\rm exc}(X, B_{j-1}) \ge 0$,
\begin{equation}\label{sternstern}
W = |X\cap B_{j-1}| \ge N \lambda^d(B_{j-1}) \ge Nv.
\end{equation}
Put 
$$p:= W/N.$$ 
We now want to check that the conditions on $p$ and $n$ in Theorem \ref{HypGeomBinom} and Lemma \ref{Lemma_Benji} hold.
Due to $B_{j-1}\subseteq B_{1}$ and ${\rm exc}(X,B_1) =0$ we have $p\le 1/4$.  Furthermore, we have $n= Nc/d \ge 20$ and, due to \eqref{sternstern} and \eqref{stern}, $p\ge v\ge 1/6 \ge 1/n$.
This leads to 
\begin{equation*}
np(1-p) \ge 20 \cdot \frac{1}{6} \left( 1 -\frac{1}{4} \right) = \frac{5}{2} > 1. 
\end{equation*}
Hence we may apply 
Theorem \ref{HypGeomBinom} and Lemma \ref{Lemma_Benji} to obtain
\begin{equation}\label{prob_est}
\begin{split}
{\bf {\rm P}} \! \left( Y_j \le np - \frac{1}{2}\sqrt{np} \right) &\ge B(n,p)\!\left( \left[ 0, np - \frac{1}{2}\sqrt{np} \right] \right) 
- \delta \big( H(N, W, n),  B(n,p) \big)\\
&\ge \frac{3}{160} - \frac{c}{d}\\ 
&\ge \frac{1}{80}. 
\end{split}
\end{equation}
If 
\begin{equation*}
Y_j = | X\cap C_j | \le np - \frac{1}{2} \sqrt{np},
\end{equation*}
then put $x_j:= 1-c/d$, and otherwise put $x_j:=1$.
We define 
\begin{equation*}
B_{j} := I \times \prod^{j}_{i=2}[0,x_i) \times [0,1)^{d-j}.
\end{equation*}
Before we go on, let us make a helpful observation: Put 
$$\eta_i: = 1_{[x_i = 1-c/d]}(X) \hspace{3ex} \text{for $i= 2,\ldots,j$.}$$ 
Then $\eta_i$
is a Bernoulli random variable and \eqref{prob_est} says that ${\rm P}(\eta_j =1) \ge 1/80$. Actually, due to our construction we proved a slightly stronger result, namely:
\begin{equation}\label{prob_est+}
{\rm P}\left( \eta_j =1 \,|\, \eta_2=v_1, \ldots, \eta_{j-1}=v_{j-2}\right) \ge 1/80 
\hspace{2ex}\text{for all $v
\in \{0,1\}^{j-2}$}
\end{equation}
(since \eqref{prob_est} holds for all values of $\eta_2,\ldots,\eta_{j-1}$ that have been determined previously in the course of the construction of $B_j$).

We now want to estimate the excess of points of $X$ in $B_j$. In the case $x_j= 1-c/d$ we have
$\lambda^d(B_{j}) = (1-c/d) \lambda^d(B_{j-1})$ and thus
\begin{equation*}
\begin{split}
{\rm exc}(X,B_j) &= |X\cap B_{j-1}| - |X \cap C_{j}| - N (1-c/d) \lambda^d(B_{j-1})\\
&\ge |X\cap B_{j-1}| - np + \frac{1}{2} \sqrt{np} - N (1-c/d) \lambda^d(B_{j-1})\\
&= (1-c/d) \left( |X\cap B_{j-1}| - N \lambda^d(B_{j-1}) \right) + \frac{1}{2} \sqrt{np}\\
&= (1-c/d) {\rm exc}(X, B_{j-1}) +  \frac{1}{2} \sqrt{W\frac{c}{d}}\\
&\ge (1-c/d) {\rm exc}(X, B_{j-1}) +  \frac{\sqrt{cv}}{2}\sqrt{\frac{N}{d}},
\end{split}
\end{equation*}
where in the last step we used \eqref{sternstern}.

In the case $x_j = 1$ we obviously have $B_j = B_{j-1}$ and consequently ${\rm exc}(X,B_j) = {\rm exc}(X, B_{j-1})$. 

Put
\begin{equation*}
k= k(X) := |\{ i \in \{2,\ldots,d\} \,|\, x_i = 1-c/d\}|.
\end{equation*}
Due to $(1-c/d)^k \ge 5v\,$ (cf. \eqref{stern}) we obtain
\begin{equation}\label{exc_est}
{\rm exc}(X,B_d) \ge  k(1-c/d)^k \frac{\sqrt{cv}}{2} \sqrt{N/d} \ge \frac{5}{2}\sqrt{cv^3}\, k \sqrt{N/d}.
\end{equation}
Thus we get on the one hand from \eqref{karo}
\begin{equation*}
\begin{split}
{\bf {\rm E}}[D^*_N(X)] &\ge \frac{1}{N}{\bf {\rm E}}[{\rm exc}(X,B_d)]\\
&\ge \sum^{d-1}_{\kappa=0}\frac{5}{2}\sqrt{cv^3}\, \kappa \sqrt{1/Nd}\, {\bf {\rm P}}(k(X) = \kappa)\\
&= \frac{5}{2}\sqrt{cv^3} \,\sqrt{1/Nd}\, \sum^{d-1}_{\kappa=0} \kappa \,{\bf {\rm P}}(k(X) = \kappa)\\
&= \frac{5}{2}\sqrt{cv^3}\, \sqrt{1/Nd} \,{\bf {\rm E}}[k(X)]\\
&\ge (\sqrt{cv^3}/32\sqrt{2}) \sqrt{(d-1)/N},
\end{split}
\end{equation*}
where in the last step we used \eqref{prob_est} to obtain 
\begin{equation*}
{\bf {\rm E}}[k(X)] = \sum^d_{i=2} {\bf {\rm E}}[\eta_i] 
=  \sum^d_{i=2} {\bf {\rm P}} \! \left( Y_i \le np - \frac{1}{2}\sqrt{np} \right)
\ge (d-1)/80.
\end{equation*}
On the other hand we get from \eqref{exc_est} for $K:= \sqrt{cv^3}/80$ 
\begin{equation*}
\begin{split}
{\bf {\rm P}} \left( D^*_N(X) < K \sqrt{d/N} \right) &\le {\bf {\rm P}} \left( {\rm exc}(X,B_d) < K \sqrt{dN} \right)\\
&\le {\bf {\rm P}} \left(  \frac{5}{2}\sqrt{cv^3}\,k(X) \sqrt{N/d} < K \sqrt{dN} \right)\\ 
&= {\bf {\rm P}} \left( k(X) < d/200 \right)\\
&= {\bf {\rm P}} \left( \sum^d_{i=2} \eta_i < d/200 \right).
\end{split}
\end{equation*}
Let $\xi_i$, $i=2,\ldots,d$, be independent Bernoulli random variables with
\begin{equation*}
{\bf {\rm P}}( \xi_i = 1) = 1/80 
\hspace{3ex}\text{and}\hspace{3ex}
{\bf {\rm P}}( \xi_i = 0) = 79/80.
\end{equation*}
Clearly, ${\bf {\rm E}}[\xi_i] = 1/80$.
Since estimate \eqref{prob_est+} holds for each $j\in \{2,\ldots,d\}$, 
we have due to Lemma \ref{Lemma_Coupling}
\begin{equation*}
{\bf {\rm P}} \left( \sum^d_{i=2} \eta_i < d/200 \right) 
\le {\bf {\rm P}} \left( \sum^d_{i=2} \xi_i  < d/200 \right).
\end{equation*}
Hence we get  from Theorem~\ref{Chernov}
\begin{equation*}
\begin{split}
{\bf {\rm P}} \left( D^*_N(X) < K \sqrt{d/N} \right) 
&\le {\bf {\rm P}} \left( \sum^d_{i=2} (\xi_i - {\bf {\rm E}}[\xi_i])  < \left( \frac{1}{200} - \frac{1}{80} \frac{d-1}{d} \right) d \right)\\
&\le {\bf {\rm P}} \left( \sum^d_{i=2} (\xi_i - {\bf {\bf {\rm E}}}[\xi_i])  < -\frac{d}{800} \right)\\
&\le \exp \left( - \frac{2d^2}{(800)^2 (d-1)} \right)\\
&= \exp \left( - \Omega (d) \right).
\end{split}
\end{equation*}
This concludes the proof of the theorem.
\end{proof}

\subsection*{Acknowledgment}
The authors thank two anonymous referees for their comments 
which helped to improve the presentation of the paper.

Part of this work was done while Michael Gnewuch was visiting the Laboratoire d'Informatique (LIX)
as  Chercheur Invit\'e of \'Ecole Polytechnique.
He likes to thank the colleagues from LIX for their hospitality.

\end{document}